\documentclass[11pt]{amsart}

\usepackage{amssymb}

\usepackage{graphicx}
\usepackage{tikz}
\usepackage{tikz-cd}

\newtheorem*{notn}{Notation}

\newtheorem{thm}{Theorem}[section]
\newtheorem{lem}[thm]{Lemma}
\newtheorem{cor}[thm]{Corollary}
\newtheorem{Q}[thm]{Question}
\newtheorem{Def}[thm]{Definition}
\newtheorem{prop}[thm]{Proposition}
\newtheorem{rem}[thm]{Remark}

\newcommand{\bdfn}{\begin{Def} \rm}
\newcommand{\edfn}{\end{Def}}

\newcommand{\tFAE}{the following are equivalent}

\newcommand{\ra}{\rightarrow}
\newcommand{\Ra}{\Rightarrow}

\newcommand{\Lglra}{\Longleftrightarrow}

\newcommand{\x}{\mathnormal{x}}

\newcommand{\ci}{\subseteq}

\newcommand{\al}{\alpha}

\newcommand{\de}{\delta}
\newcommand{\e}{\varepsilon}
\newcommand{\ka}{\kappa}
\newcommand{\la}{\lambda}

\newcommand{\mb}{\mathbb}

\newcommand{\iy}{\infty}

\newcommand{\beqa}{\begin{eqnarray*}}
\newcommand{\eeqa}{\end{eqnarray*}}

\newcounter{cnt1}
\newcounter{cnt2}
\newcounter{cnt3}
\newcounter{cnt4}
\newcommand{\blr}{\begin{list}{$($\roman{cnt1}$)$} {\usecounter{cnt1}
 \setlength{\topsep}{0pt} \setlength{\itemsep}{0pt}}}
\newcommand{\blR}{\begin{list}{\Roman{cnt4}.\ } {\usecounter{cnt4}
 \setlength{\topsep}{0pt} \setlength{\itemsep}{0pt}}}
\newcommand{\bla}{\begin{list}{$(\alph{cnt2})$} {\usecounter{cnt2}
 \setlength{\topsep}{5pt} \setlength{\itemsep}{5pt}}}
 \newcommand{\blaa}{\begin{list}{$(\alph{cnt2}')$} {\usecounter{cnt2}
 \setlength{\topsep}{5pt} \setlength{\itemsep}{5pt}}}
\newcommand{\bln}{\begin{list}{$($\arabic{cnt3}$)$} {\usecounter{cnt3}
 \setlength{\topsep}{0pt} \setlength{\itemsep}{0pt}}}
\newcommand{\el}{\end{list}}

\begin{document}

\title[Extensions of $p$-compact operators]{Extensions of $p$-compact operators in Banach spaces}

\author[Karak]{Sainik Karak}
\author[Paul]{Tanmoy Paul}
\address{Department of Mathematics,
 Indian Institute of Technology, Hyderabad, India}

\email{\tt ma22resch11001@iith.ac.in \& tanmoy@math.iith.ac.in}

\subjclass[2000]{Primary 46B20, 46B10, 46B25.}

\keywords{Compact operators, $p$-compact operators, $p$-approximation property, $K_p(X,Y), K_p^d(X,Y)$. \hfill \textbf{\today}}

\begin{abstract}
We analyze various consequences in relation to the extension of operators $T:X\to Y$ that are $p$-compact, as well as the extension of operators $T:X\to Y$ whose adjoints $T^*:Y^*\to X^*$ are $p$-compact. In most cases, we discuss these extension properties when the underlying spaces, either the domain or codomain, are $P_\lambda$ spaces. We also address whether these extensions are almost norm-preserving in such circumstances where the extension $\widetilde{T}$ of $T$ exists.  It is observed that an operator can often be extended to a larger domain when the codomain is appropriately extended as well. Specific assumptions might enable us to obtain an extension of an operator that maintains the same range. In this context, both necessary and sufficient conditions are established for a Banach space to qualify as a $L_1$-predual.
\end{abstract}

\maketitle

\section{Introduction}
\subsection{Objectives}

In this paper, we address the following questions. We refer to the next section for the necessary definitions of any terms not explained here.

\begin{Q}\label{Q1}
\bla
Let $X$ be a Banach space and $1\leq p<\iy$.
\item Let $T:X\to Y$ be a $p$-compact (weakly $p$-compact) operator and $Z\supseteq X$. Does a $p$-compact (weakly $p$-compact) extension $\widetilde{T}:Z\ra Y$ exist such that $\ka_p(\widetilde{T})=\ka_p(T)$ ($\omega_p(\widetilde{T})\leq\omega_p(T)$)? 
\item Let $T:X\to Y$ be a bounded linear operator, and $Z\supseteq X$. Assume that $T^*:Y^*\ra X^*$ is $p$-compact. Does an extension $\widetilde{T}:Z\ra Y$ exist such that $\widetilde{T}^*:Y^*\ra Z^*$ is $p$-compact and $\ka_p^d(\widetilde{T})=\ka_p^d(T)$? 
\el
\end{Q}

\subsection{Preliminaries}

We introduce the following notations, which are required to define the central theme of this article. Here $X$ denotes a complex Banach space and $(x_n)$ represents a sequence in $X$. $B_X$ and $S_X$ represent the closed unit ball and the unit sphere of $X$, respectively. $B(X,Y)$ and $K(X,Y)$ represent the space of all bounded and compact linear operators from $X$ to $Y$, respectively. $F(X,Y)$ represents the set of all finite-rank linear operators from $X$ to $Y$.

\begin{notn}
\bla 
\item Define $\ell_p^s(X)=\{(x_n)\in \oplus_{n=1}^\iy X:\sum_n\|x_n\|^p<\iy\}$, for $1\leq p<\iy$.
\item Define $\ell_p^w(X)=\{(x_n)\in\oplus_{n=1}^\iy X: \sum_n|x^*(x_n)|^p<\iy, x^*\in X^* \}$, for $1\leq p<\iy$. 
\el
\end{notn}

When $(x_n)\in \ell_p^s(X)$ ($(x_n)\in\ell_p^w(X)$), we define the norms,
\[\|(x_n)\|_p^s=\Big(\sum_{n=1}^\iy\|\x_n\|^p\Big)^{\frac{1}{p}} \mbox{~and~}
\]
\[
\|(x_n)\|_p^w=\sup\{\Big(\sum_n|x^*(x_n)|^p\Big)^{\frac{1}{p}}:x^*\in B_{X^*}\}  
\]
respectively, such that $(\ell_p^s(X),\|.\|_p^s)$ and $(\ell_p^w(X),\|.\|_p^w)$ form complete normed linear spaces.
For a given $x=(x_n)\in\ell_p^s(X)$ (or $(x_n)\in\ell_p^w(X)$) one can define $E_x:\ell_q\ra X$, a bounded linear operator by $E_x (\al_n)=\sum_n\al_nx_n$. With this identification $x\mapsto E_x$, $\ell_p^w(X)\cong B(\ell_q,X)$, $\frac{1}{p}+\frac{1}{q}=1$ and $\ell_1^w(X)\cong B(c_0,X)$ (see \cite{JD}). It is clear that $\ell_p^s(X)\ci K(\ell_q,X)$, which justifies $\ell_p^s(X)\ci \ell_p^w(X)$. We refer the reader to \cite[p.34]{JD} for more details on these identifications.

Alexander Grothendieck has made it well-known that a relatively compact set in a Banach space can be found in the convex hull of a null sequence (see \cite[p.112]{AG1}).
Motivated by Grothendieck's result, Karn and Sinha introduced the notion of a (weakly) $p$-compact set for $1\leq p\leq\iy$.

\bdfn\label{D1}
Let $K\ci X$. For $x=(x_n)$, consider $E_x:\ell_q\ra X$ as stated above.
\bla
\item $K$ is said to be {\it relatively $p$-compact}, $1\leq p\leq \iy$, if there exists $x=(x_n)\in\ell_p^s(X)(\,1\leq p<\iy\,)\;(\,x\in c_0^s(X)$ if $p=\iy\,)$ such that $K\ci E_x(B_{\ell_q})$.
\item $K$ is said to be {\it relatively weakly $p$-compact}, $1\leq p\leq \iy$, if there exists $x=(x_n)\in\ell_p^w(X)(\,1\leq p<\iy\,)\;(\,x\in c_0^w(X)$ if $p=\iy\,)$ such that $K\ci E_x(B_{\ell_q})$. 
\el
\edfn

According to this description, $\iy$-compact sets are precisely the compact sets. Moreover, every $p$-compact set is $q$-compact whenever $1\leq p<q\leq\iy$; however, in general, $q$-compact sets are not necessarily $p$-compact. We also note that $c_0^s(X)=c_0^w(X)$ when $X=\ell_1$. By contrast, for $1\leq p<\iy$, $\ell_p^s(X)\subsetneqq \ell_p^w(X)$, whenever $X$ is infinite dimensional and vice versa.
One can now generalize the notion of the $p$-compact operator in the following sense. 

\bdfn\label{D2}
For Banach spaces $X, Y$ and $1\leq p\leq\iy$, an operator $T\in B(X, Y)$ is called $p$-compact (weakly $p$-compact) if $T$ maps bounded subsets of $X$ to relatively $p$-compact (weakly $p$-compact) subsets of $Y$. In other words, there exists $y\in\ell_p^s(Y)$ ($y\in \ell_p^w (Y)$) for $p<\iy$ (for $p=\iy$, $y\in c_0^s(Y)$) such that $T(B_X)\ci E_y(B_{\ell_q})$, where $\frac{1}{p}+\frac{1}{q}=1$. 
\edfn

\begin{notn}
For Banach spaces $X, Y$, we define the following:
\bla
\item $K_p(X,Y)=\{T\in B(X,Y): T \mbox{~is~}p-\mbox{compact}\}$.
\item $W_p(X,Y)=\{T\in B(X,Y): T \mbox{~is weakly~}p- \mbox{compact}\}$.
\el
\end{notn}

$K_p(X,Y)$ and $W_p(X,Y)$ are Banach operator ideals with respect to some suitable norms $\ka_p$ and $\omega_p$ respectively. For a given operator $T$, $\ka_p(T) (\omega _p(T))$ depends on the factorization of the operator $T$ through a quotient space of $\ell_q$, $\frac{1}{p}+\frac{1}{q}=1$.
Moreover, if $(A,\al)$ is an operator ideal for Banach spaces, one can define $A^d(X,Y)=\{T\in B(X,Y):T^*\in A(Y^*,X^*)\}$. For $T\in A^d(X,Y)$, we define $\al^d(T)=\al(T^*)$. Then $(A^d,\al^d)$ is again an operator ideal and is called the {\it dual ideal} of $(A,\al)$. It is well-known that $A^d$ is a Banach operator ideal whenever $(A,\al)$ is also a Banach operator ideal. In this paper, we discuss various extension properties of the dual ideal.

\begin{notn}
For Banach spaces $X, Y$, we define the following:
\[
K_p^d(X,Y)=\{T\in K_p(X,Y):T^*\in K_p(Y^*,X^*)\}.
\]
\end{notn}

Interested readers can refer to \cite{SK,SK1} for more details on these ideals.
We now move on to the notion of {\it approximation property} in Banach spaces.

\bdfn
A Banach space $X$ is said to have the {\it metric approximation property} (MAP) if for every compact subset $K$ of $X$ and for $\e>0$ there exists an operator $T$ with a finite-dimensional range from $X$ into itself such that $\|T\|=1$ and $\|Tx-x\|\leq \e$ for every $x\in K$.
\edfn

Hence if $K\ci X$ is compact, then $p_K:B(X,Y)\ra \mb{R}$ defines a seminorm, where $p_K(T)=\sup_{x\in K} \|Tx\|$. Thus, if $\tau$ represents the topology induced by the seminorms $\{p_K:K\ci X \mbox{~compact}\}$ then the identity on $X$, $I\in \overline{F(X)}^\tau$.

We now turn our focus towards the extension properties of Banach spaces.

\bdfn\cite{MD}\label{D3}
\bla
\item A Banach space $X$ is said to be a {\it $P_\la$-space}, for some $\la\geq 1$, if for any Banach space $Z\supseteq X$ (with $X$ as a subspace) there exists a projection $P:Z\ra X$, with $\|P\|\leq \la$. 
\item A Banach space $X$ is said to be injective if for any Banach space $Z$ and any subspace $Y$ of $Z$, every  bounded linear operator $T:Y\ra X$ admits an extension $\widetilde{T}:Z\ra X$ such that $\|T\|=\|\widetilde{T}\|$.
\el
\edfn

It is well known that the spaces $P_1$ are injective Banach spaces, and vice versa.
In \cite[p.94]{MD} the author discusses the $P_\la$ spaces and demonstrates that these spaces provide Hahn-Banach-type extensions for linear operators in Banach spaces. The case $\la=1$ is of particular interest: the family of $P_1$ spaces also known as {\it Banach spaces with the extension property} \cite[p.2]{JL}. Lindenstrauss initiated this investigation systematically in his memoir \cite{JL}. It is widely recognized that real $P_1$-spaces are those Banach spaces that are of the form $C_{\mb{R}}(\Omega)$ for some Stonean space  $\Omega$ (see \cite{DG, JK, LN}). In \cite[Theorem~2]{HA} Hasumi observed that a complex Banach space has the extension property if and only if it is isometric to $C(\Omega)$, for some Stonean space $\Omega$.

 Grothendieck \cite{AG} showed that if $X$ is a real Banach space, then $X^*$ is isometric to an $L_1$-space if and only if $X^{**}$ is a $P_1$ space. The results of Sakai (see \cite{SS}) show that this theorem is also valid for complex Banach spaces.

\bdfn\label{D4}
A Banach space $X$ is said to be an {\it $L_1$-predual} if $X^*\cong L_1(\mu)$ for some measure space $(S,\Sigma,\mu)$.
\edfn

The class of Banach spaces whose duals are $L_1$-spaces is a well-studied object in functional analysis. We refer to Chapters 6 and 7 of Lacey's monograph \cite{HL} for characterizations of these spaces and their properties.
All such spaces with real scalars can be characterized by intersection properties of closed balls (see \cite[p.212]{HL}). However, these intersection properties may fail for complex $L_1$-predual spaces. Complex Banach spaces that are $L_1$-preduals are {\it $E$-spaces} and vice versa (see \cite[Theorem~4.9]{HU}). It is well-known that every $L_1$-predual has the MAP.

Recall a result by Lindenstrauss, stated in \cite[Theorem~2.1]{JL}. The result was derived for real scalars, however, similar observations are also valid for complex scalars. In the subsequent sections, we assume that \cite[Theorem~2.1]{JL} holds for complex scalars.

Note that in finite-dimensional spaces, every compact set is $p$-compact, for $1\leq p\leq\iy$. 
Moreover, if $T$ is a finite rank operator between Banach spaces $X$ and $Y$, then $T=\sum_{i=1}^n x_i^*\otimes y_i$, for some $x_i^*\in X^*$ and $y_i\in Y$. Then $T(B_X)\ci E_v(B_{\ell_1(n)})$, for a suitable $v=(v_i)_{i=1}^n$, $v_i\in Y$. This implies that $T$ is a $p$-compact operator for $1\leq p\leq \iy$.

\bdfn\cite{JD}
Suppose that $1\leq p<\iy$ and that $T: X\ra Y$ is a linear operator between Banach spaces. We say that $T$ is $p$-summing if there exists a constant $c\geq 0$ such that for $m\in\mb{N}$ and for every choice of $x_1,...,x_m$ in $X$ we have $\Big(\sum_{i=1}^m\|Tx_i\|^p\Big)^\frac{1}{p}\leq c.\sup\Big\{\Big(\sum_{i=1}^m | x^*(x_i)|^p\Big)^\frac{1}{p}:x^*\in B_{X^*}\Big\}$.
\edfn

The least $c$ for which this inequality always holds is denoted by $\pi_p(T)$.

\begin{notn}
For Banach spaces $X, Y$, we define
\[
\Pi_p(X,Y)=\{T\in B(X,Y):T \mbox{~is a~}p-\mbox{summing operator}\}.
\]
\end{notn}

We use the techniques by Lindenstrauss in \cite{JL} in order to study the extensions of $p$-compact operators. In this investigation, the study of the operator ideals $K_p(X,Y)$ and $K_p^d(X,Y)$ in \cite{SK, SK1} is also used. In particular, when $T\in K_p^d(X,Y)$ we use the decomposition of $T$ derived in \cite{SK1} to estimate $\ka_p^d(\widetilde{T})$, for an extension $\widetilde{T}$ of $T$.

\subsection{Statements of the main results}
In this note, we address question~\ref{Q1} in the following sense. We assume in Theorems~\ref{T1.1} and \ref{T1.2}, we assume that $X, Y, Z$ are Banach spaces with $Z\supseteq X$ and that $X^{**}$ is a $P_\la$-space for some $\la\geq1$. Suppose that $1< p<\iy$.

\begin{thm}\label{T1.1}
Suppose that $T\in K_p(X,Y)$ $(T\in W_p(X,Y))$. Then there exists $\widetilde{T}\in K_p(Z,Y)$ $(\widetilde{T}\in W_p(Z,Y))$ such that $\ka_p(\widetilde{T})\leq \la\ka_p(T)$ $(\omega_p(\widetilde{T})\leq\la\omega_p(T))$.
\end{thm}

\begin{thm}\label{T1.2}
Suppose that $T\in K_p^d(X,Y)$. Then there exists $\widetilde{T}\in K_p^d(Z,Y)$ such that $\ka_p^d(\widetilde{T})\leq \la\ka_p^d(T)$.
\end{thm}

It follows that, if a Banach space $X$ has the MAP and every compact (weakly compact) operator $T:X\ra Y$ admits a compact (weakly compact) extension $\widetilde{T}:Z\ra Y$ with $\|\widetilde{T}\|\leq\la\|T\|$, then for $1<p<\iy$ every $p$-compact (weakly $p$-compact) operator $S:X\ra Y$ has a $p$-compact (weakly $p$-compact) extension $\widetilde{S}:Z\ra Y$ such that $\ka_p(\widetilde{T})\leq\la\ka_p(T)$ ($\omega_p(\widetilde{T})\leq\la\omega_p(T)$) (see Corollary~\ref{C1}). Here $Y, Z$ are any Banach spaces such that $Z\supseteq X$. A partial converse to this result is obtained in Corollary~\ref{C4}.

\section{Extension of $T\in K_p(X,Y)$}

We begin this section by observing that a compact operator admits a norm-preserving compact extension by suitably enlarging its codomain.

\begin{prop}\label{P9}
Let $X, Y$ be Banach spaces and $T\in K(X, Y)$. Suppose $Z\supseteq X$ such that $\dim(Z/X)<\iy$. Then, there exists $V\supseteq Y$ such that $T$ has a compact extension $\widetilde{T}:Z\to V$ with $\|T\|=\|\widetilde{T}\|$.
\end{prop}

\begin{proof}
This follows from \cite[Lemma 1.1]{JL}.
\end{proof}

We note that a similar conclusion to Proposition~\ref{P9} holds if $T:X\ra Y$ is weakly compact. Moreover, it is clear that the sum of two $p$-compact (weakly $p$-compact) sets is again $p$-compact (weakly $p$-compact). Hence, we obtain the following.

\begin{prop}\label{P8}
Let $X, Y$ be Banach spaces and $T\in K_p(X, Y)$ ($T\in W_p(X,Y)$). Let $Z\supseteq X$ be such that $\dim(Z/X)<\iy$. Then, there exists $V\supseteq Y$ such that $T$ admits a $p$-compact (weakly $p$-compact) extension $\widetilde{T}:Z\to V$ for $1\leq p\leq\iy$. 
\end{prop}

We do not know whether the norm $\ka_p$ (or $\omega_p$) of the operator $\widetilde{T}$ stated in Proposition~\ref{P8} can be preserved. 

\begin{cor}\label{C2}
Let $T\in K_p(X,Y)$ admit a bounded extension $\widetilde{T}:Z\to Y$, where $\dim (Z/X)<\iy$. Then, $\widetilde{T}\in K_p(Z,Y)$ for $1\leq p\leq\iy$.
\end{cor}

We now assume a reflexive space, and hence all its quotients are conjugate spaces. Our next result uses the equivalence $(1)\Lglra (5)$ in Theorem~2.1 of \cite[p.11]{JL}.

\begin{thm}\label{P2}
Let $X$ be a Banach space such that $X^{**}$ is a $P_\la$ space. Suppose that $1< p\leq\iy$, and let $Z$ be a Banach space with $X\ci Z$.
\bla
\item If $T\in K_p(X,Y)$, then there exists $\widetilde{T}\in K_p(Z,Y)$ such that $\ka_p(\widetilde{T})\leq \la\ka_p(T)$.
\item If $T\in W_p(X,Y)$, then there exists $\widetilde{T}\in W_p(Z,Y)$ such that $\omega_p(\widetilde{T})\leq \la\omega_p(T)$.
 \el
\end{thm}
\begin{proof}
$(a).$ The case for $p=\iy$ follows from \cite[p.11]{JL}, it remains to consider $1< p<\iy$.

From \cite[Theorem~3.2]{SK}, we obtain $y\in \ell_p^s(Y)$ such that $T=\widetilde{E_y}\circ T_y$, where $T_y:X\ra \ell_{p'}/N_y$ is a bounded linear and $\widetilde{E_y}:\ell_{p'}/N_y\ra Y$ is a compact linear, $\frac{1}{p}+\frac{1}{p'}=1$. By assumption, we obtain $\widetilde{T_y}:Z\ra \ell_{p'}/N_y$ a bounded linear, such that $\|\widetilde{T_y}\|\leq \la \|T_y\|\leq\la$ [since$\|T_y\|\leq 1$]. Define $\widetilde{T}=\widetilde{E_y}\circ \widetilde{T_y}$.
  
{\sc Claim:~}$\widetilde{T}$ is $p$-compact and $\ka_p(\widetilde{T})\leq\la \ka_p(T)$.

Indeed, there exists $y\in \ell_p^s(Y)$ such that $\widetilde{T}(B_Z)\ci E_y(B_{\ell_{p'}})$. Let $z\in B_Z$, then $\widetilde{T_y}(z)\in \ell_{p'}/N_y$ i.e., there exists $\al\in\ell_{p'}$ with $\widetilde{T_Y}(z)=\al+N_y$. Since $\|\widetilde{T_y}\|\leq \la$, we get $\|\al+N_y\|\leq \la$.

Because $N_y$ is a closed subspace of the reflexive and strictly convex space, $N_y$ is Chebyshev in $\ell_{p'}$. Thus, there exists a unique $\beta\in N_y$ such that $\|\al-\beta\|_{p'}=\|\al+N_y\|\leq\la$. Hence $\al-\beta\in \la B_{\ell^{p'}}$.

Now consider $\widetilde{T_y}(z)=\al-\beta+N_y$. Therefore, $\widetilde{E_y}(\al-\beta+N_y)\in \la E_y(B_{\ell_{p'}})=E_{\la y}(B_{\ell_{p'}}), \la y\in \ell_p^s(Y)$. In this way, we obtain $\widetilde{T}(B_Z)\ci E_{z}(B_{\ell^{p'}}), z=\la y$.

Observe that $\|\la y\|_p^s=\la \|y\|_p^s$. This yields $\ka_p(\widetilde{T})\leq\la \ka_p(T)$.

$(b)$ The argument is analogous to $(a)$ with $\|y\|_p^w$ in place of $\|y\|_p^s$ when evaluating $\omega_p(\widetilde{T})$.
\end{proof}

In \cite[Theorem~3.1]{CK} Choi and Kim have derived that $\widetilde{E_y}$ is a $p$-compact operator. Hence, $\widetilde{T}$ is a $p$-compact operator from by definition.

\begin{cor}\label{C1}
Let $X$ be a Banach space with MAP and $\la\geq 1$. Let $Y, Z$ be Banach spaces such that $Z\supseteq X$. Suppose that for every compact (weakly compact) operator $T:X\ra Y$ has a compact (weakly compact) extension $\widetilde{T}:Z\ra Y$ with $\|\widetilde{T}\|\leq\la\|T\|$. Then, every $p$-compact (weakly $p$-compact) operator $S:X\ra Y$ has a $p$-compact (weakly $p$-compact) extension $\widetilde{S}:Z\ra Y$ with $\ka_p(\widetilde{S})\leq\la\ka_p(S)$, where $1<p\leq\iy$. 
\end{cor}
\begin{proof}
Note that if every compact operator $T:X\ra Y$ has a compact extension $\widetilde{T}:Z\ra Y$ with $\|\widetilde{T}\|\leq\la\|T\|$ and $X$ has MAP then $X^{**}$ is a $P_\la$ space and hence the result follows from Theorem~\ref{P2}.
\end{proof}

We now derive a sufficient condition for the spaces which are $L_1$-preduals. The main result we derive in this connection is that Theorem~\ref{T5} uses a characterization of $L_1$-preduals of Banach spaces under real scalars.

\begin{lem}\label{L1}
Let $(E,\|.\|)$ be a finite-dimensional Banach space and $K$ be a relatively compact set in $E$. Then for $\e>0$ there exists a finite set $\{x_1,x_2,...,x_k\}\subset E$ such that $K\subseteq conv\{x_1,...,x_k\}$ and $\sup\{\|x_i\|:1\leq i\leq k\}<\sup\{\|k\|:k\in K\}+\e$. 
\end{lem}
\begin{proof}
Suppose that $\dim E=n$. Hence, there exists a basis $(e_i)_{i=1}^n$ of $E$ such that $\|e_i\|=1$, $1\leq i\leq n$. We induce $\|x\|_\infty=\max_i|\alpha_i|$, where $x=\sum_i\alpha_ie_i$.
    
Now there exist $c_1, c_2>0$ such that $$c_1\|x\|\leq\|x\|_\infty\leq c_2\|x\|, \text{ for all }x\in E.\quad\quad(1)$$

We denote a ball in $E$ centered at $x$ and radius $r$ with respect to the norms $\|.\|$ and $\|.\|_\infty$ by $B_E(x,r)$ and $B_\infty(x,r)$, respectively.
 Now for $\varepsilon>0$, choose $\delta>0$ such that $\delta<\frac{\varepsilon c_1}{c_2}$. For this $\delta$ there exist $x_1,x_2,...,x_m\in K$ such that $$\begin{aligned}
        K&\subseteq \bigcup_{i=1}^m B_E(x_i,\delta)\\
        &\subseteq \bigcup_{i=1}^m B_{\infty}(x_i,c_2\delta)\text{ [ by (1) ]}.
    \end{aligned}$$
    Now observe that each $B_\infty(x_i,c_2\delta)$ has $2^n$ extreme points $\{x_i^1,...,x_i^{2^n}\}$ and $B_{\infty}(x_i,\delta c_2)=conv\{x_i^1,...,x_i^{2^n}\}$. It follows that $$K\subseteq conv\{x_1^1,...,x_1^{2^n},...,x_m^1,...,x_m^{2^n}\}.$$

Also note that for a fixed $i$ and for $1\leq j\leq 2^n$, $\|x_i-x_i^j\|_\infty\leq \delta c_2$ and hence, $\|x_i-x_i^j\|\leq \frac{\delta c_2}{c_1}.$ It follows that $\|x_i^j\|\leq \|x_i\|+\frac{\delta c_2}{c_1}$. Since each $x_i\in K$, we have
    $$\begin{aligned}
    \sup\{\|x_i^j\|:1\leq i\leq m\text{ and }1\leq j\leq 2^n\}&\leq \sup\{\|k\|:k\in K\}+\frac{\delta c_2}{c_1}\\
    &<\sup\{\|k\|:k\in K\}+\varepsilon.
    \end{aligned}$$
    Hence the result follows.
\end{proof}

For a Banach space $Y$, by $c_{00}^s(Y)$, we denote the set of all finitely supported sequences in $Y$ with the usual supremum norm. 

\begin{rem}\label{R1}
Note that if $T$ is a finite rank operator, then by Lemma~\ref{L1}, $\|T\|=\inf\{\|y\|_\iy: T(B_X)\ci E_y(B_{\ell_1}) \mbox{~and~} y\in c_{00}^s(Y)\}$.
\end{rem}

Recall that $\lim_{p\to\infty}\|(\alpha_i)\|_p=\|(\alpha_i)\|_\infty$ for $(\alpha_i)\in\mathbb{K}^n$. 
Also, recall the following from \cite[Proposition~3.15]{DPS}

\begin{thm}
Let $X,Y$ be Banach spaces and $T\in K_p(X,Y)$ for some $p\geq 1$. Then $\kappa_p(T)=\inf\big\{\|y\|_p^s:T(B_X)\subseteq E_y(B_{\ell_{p'}})\big\}$.
\end{thm}

\begin{thm}\label{T1}
Assume $T\in B(X,Y)$ is of finite rank, then $\lim\limits_{p\to\infty}\kappa_p(T)=\|T\|.$
\end{thm}
\begin{proof}
 Let $\varepsilon>\delta>0$. Then by Remark~\ref{R1}, we can choose $y\in c_{00}^s(Y)$ such that $\|y\|_\infty<\|T\|+\varepsilon-\delta$. Since $\|y\|_p\to\|y\|_\infty$ as $p\to\infty$, hence for $\delta>0$, choose $p$ such that $\|y\|_p<\|y\|_\infty+\delta$. This implies that $\|y\|_p<\|T\|+\varepsilon$ and hence $\kappa_p(T)<\|T\|+\varepsilon$. Now the proof follows from the fact that $\|T\|\leq \kappa_q(T)\leq\kappa_p(T)$ for $1\leq p<q<\infty$.  
\end{proof}

\begin{thm}\label{T5}
Let $X$ be a real Banach space, $p>1$ and $\e>0$. Suppose that for all $q\geq p$ and every operator $T:Y\ra X$ with $\dim T(X) \leq 3$ there exists an extension $\widetilde{T}:Z\ra X$ where $Z\supseteq Y$ and $\dim Z/Y=1$ such that $\ka_q(\widetilde{T})\leq (1+\e)\ka_q(T)$. Then $X$ is an $L_1$-predual.
\end{thm}

\begin{proof}
From the assumption on $T$, it follows from Theorem~\ref{T1} that $\|\widetilde{T}\|\leq (1+\e)\|T\|$. The result now follows from \cite[Theorem~5.4]{JL}.
\end{proof}

Theorem~\ref{T5} gives a partial converse to Corollary~\ref{C1}.

\begin{cor}\label{C4}
Let $X$ be a Banach space, $p>1$ and $\e>0$. Let $Y, Z$ be Banach spaces such that $Z\supseteq Y$.
Suppose that for all $q>p$, $T\in K_q(Y,X)$ has an extension $\widetilde{T}\in K_q(Z,X)$ such that $\ka_q(\widetilde{T})\leq (1+\e)\ka_q(T)$. Then for all compact (weakly compact) $T:Y\ra X$, there exists a compact (weakly compact) extension $\widetilde{T}:Z\ra X$ such that $\|\widetilde{T}\|= \|T\|$.
\end{cor}

Similar to $\ell_p^s(X)$ as stated in Section~2, we define $\oplus_{c_0}Y_n=\{(y_n): y_n\in Y_n, \lim_n\|y_n\|=0\}$, for a family of Banach spaces $(Y_n)_{n=1}^\iy$.

\begin{thm}\label{T9}
\bla
\item Let $X$ be a Banach space such that every $T\in K_p(Y,X)$ has an extension $\widetilde{T}\in K_p(Z,X)$, where $Z\supseteq Y$. Then there is a constant $\eta$ so that for every such $Y,Z$ and $T$ there exists a $p$-compact extension $\widetilde{T}$ with $\ka_p(\widetilde{T})\leq \eta\ka_p(T)$ for $1\leq p\leq\iy$.
\item Let $X$ be a Banach space such that every $T\in K_p(X,Y)$ has an extension $\widetilde{T}\in K_p(Z,Y)$, where $(Z\supseteq X)$. Then there is a constant $\eta$ such that for every such $Y,Z$ and $T$ there is a $p$-compact extension $\widetilde{T}$ with $\ka_p(\widetilde{T})\leq \eta\ka_p(T)$ for $1<p\leq\iy$.
\el
\end{thm}
\begin{proof}
$(a).$ Suppose no such $\eta$ exists. Then for every $n$ there are spaces $Z_n\supseteq Y_n$ and a $p$-compact operator $T_n$ from $Y_n$ to $X$ with $\ka_p(T_n)=1$ such that any $p$-compact extension $\widetilde{T_n}$ of $T_n$ from $Z_n$ to $X$ satisfies $\ka_p(\widetilde{T_n})\geq n^3$. Let $Y=\oplus_{c_0}Y_n$ and define $T:Y\to X$ by $T=\sum_{n=1}^\iy \frac{T_n'}{n^2}$, 
where $T_n':Y\ra X$ is defined by $T_n'\big((y_1,\ldots,y_n,\ldots)\big)=T_n(y_n)$. Since $T_n'(B_Y)=T_n(B_{Y_n})$, it follows that $\ka_p(T_n')=\ka_p(T_n)=1$.

Now we have, $\ka_p(T)\leq\sum\frac{\ka_p(T_n')}{n^2}=\sum\frac{1}{n^2}<\iy$.
Hence $T\in K_p(Y,X)$. Let $\widetilde{T}$ be a $p$-compact extension of $T$ from $\oplus_{c_0}Z_n$ to $X$. Then the restriction of $n^2\widetilde{T}$ to $Z_n$ (i.e. to the sequences $(0,...,z_n,0,...)$) is an extension of $T_n$. From our assumption $\ka_p(n^2\widetilde{T})\geq n^3$, which leads to $\ka_p(\widetilde{T})\geq n$, for all $n$, which is a clear contradiction.

$(b).$ We first claim the following.

{\sc Claim:~} There exists a $P_1$-space $W$, $W \supseteq X$, such that for any Banach space $Y$ and $T \in K_p(X,Y)$ there exists an extension $\widetilde{T} \in K_p(W,Y)$ such that $\ka_p(\widetilde{T}) \leq \eta \kappa_p(T)$.

Suppose that no such $\eta$ exists. Then, for every $n$ there exists a $P_1$ space $W_n \supseteq X$, a Banach space $Y_n$ and a $p$-compact operator $T_n$ from $X$ to $Y_n$ with $\ka_p(T_n) = 1$ such that any $p$-compact extension $\widetilde{T_n}$ of $T_n$ from $W_n$ to $Y_n$ satisfies $\ka_p(\widetilde{T_n}) \geq n^3$. Let $Y=\oplus_{c_0}Y_n$ and consider $T_n:X\to Y$ as each $Y_n$ is a subspace of $Y$. Now define $T : X \to Y$ by $T = \sum_{n=1}^{\iy} \frac{T_n}{n^2}$.
Clearly, $T \in K_p(X,Y)$ because each $T_n \in K_p(X,Y)$. Now observe that $\bigoplus_{\ell_\iy} W_n = W$ is a $P_1$ space as each $W_i$ is a $P_1$ space and $W$ contains $X$. By our hypothesis there exists a $p$-compact extension $\widetilde{T}$ of $T$ from $W$ to $Y$. Then the restriction of $n^2 \widetilde{T}$ to $W_n$ is an extension of $T_n$. From our assumption $\ka_p(n^2 \widetilde{T}) \geq n^3$, which implies that $\ka_p(\widetilde{T}) \geq n$, for all $n$, which is a clear contradiction. Thus the claim follows.

Next, assume that $Z$ is a Banach space and $Z \supseteq X$ and let $T \in K_p(X,Y)$. From the above claim, there exists an extension $\widetilde{T} \in K_p(W,Y)$. Moreover, the identity $I : X \to X$ has an extension $\widetilde{I} : Z \to W$ with $\|\widetilde{I}\| = 1$, which follows from the property of $P_1$-space. Clearly $\widetilde{T} \circ \widetilde{I} : Z \to Y$ is a $p$-compact extension of $T$ and finally $\ka_p(\widetilde{T} \circ \widetilde{I}) \leq \ka_p(\widetilde{T}) \|\widetilde{I}\| \leq \eta \ka_p(T)$.
\end{proof}

\begin{rem}
Theorem~\ref{T9} also holds if we replace the $p$-compact operator by a weakly $p$-compact operator. 
\end{rem}

In the next result, it is observed that in some cases, to obtain an extension of a $p$-compact operator $T$, it suffices to find a $p$-compact operator $S$ that is close to $T$ in the sense of the $\ka_p$-norm, not necessarily an extension of $T$.

\begin{prop}\label{P1}
For a Banach space $X$ and $1\leq p\leq\iy$, \tFAE.  
\bla
\item For every Banach space $Y$, every $T\in K_p(Y,X)$ and every $\e>0$, there exists $\widetilde{T}\in K_p(Z,X)$, where $Z\supseteq Y$ such that $\ka_{p}(\widetilde{T})\leq (\la+\e)\,\ka_{p}(T)$ and $\ka_{p}(\widetilde{T}|_{Y}-T)\leq\e$.
\item For every Banach space $Y$, every $T\in K_p(Y,X)$, and every $\e>0$ there exists an extension $\widetilde{T}\in K_p(Z,X)$, where $Z\supseteq Y$ such that $\ka_{p}(\widetilde{T})\leq(\la+\e)\ka_{p}(T)$.
\el
\end{prop}
\begin{proof}
It remains to prove $(a)\Ra (b)$.

Let $Z\supseteq Y,\;\e>0$, and $T\in K_p(Y,X)$ be given. By $(a)$, there exists $\widetilde{T}_1\in K_p(Z,X)$ satisfying the following conditions.
\begin{equation}\label{eq4}
   \ka_p(\widetilde{T}_1)\leq\,(\la+\e)\,\ka_p(T),\quad \ka_p(\widetilde{T}_1|_Y-T)<\frac{\e}{2}.
\end{equation}
Now, $T-\widetilde{T}_1|_Y\in K_p(Y,X)$  so by $(a)$, there exists  $\widetilde{T}_2\in K_p(Z,X)$, satisfying
\[   
   \ka_p(\widetilde{T}_2)\leq\,(\la+1)\,\ka_p(T-\widetilde{T}_1|_Y),\quad \ka_p\big(\widetilde{T}_2|_Y-(T-\widetilde{T}_1|_Y)\big)<\frac{\e}{2^2}.
\]  
Proceeding inductively, we obtain a sequence $(\widetilde{T}_n)\ci K_p(Z,X)$ such that the inequality (\ref{eq4}) holds for $n=1$, and for $n\geq 2$ we have,
\[
    \ka_p(\widetilde{T}_n)\leq\,(\la+1)\,\ka_p\Big(T-(\widetilde{T}_1+\widetilde{T}_2+...+\widetilde{T}_{n-1})|_Y\Big),
\]
\begin{equation}\label{eq12}
    \ka_p\Big(\widetilde{T}_n|_Y-\big(T-(\widetilde{T}_1+\widetilde{T}_2+...+\widetilde{T}_{n-1})|_Y\big)\Big)<\frac{\e}{2^n}.
\end{equation}

Hence, for $n\geq 2$, we have $\ka_p(\widetilde{T}_n)\leq (\la+1)\e/2^{n-1}$. Therefore, the series $\sum\limits_{n=1}^\iy \widetilde{T}_n$ converges in the ($\ka_p$) norm topology to an operator $\widetilde{T}\in K_p(Z,X)$ satisfying $\widetilde{T}|_Y=T$. In fact, by inequality (\ref{eq12}) for $\de>0$, there exists $m$ such that $\ka_p((T-(\widetilde{T_1}+\widetilde{T_2}+\ldots+\widetilde{T_m})|_Y))<\de$.
   \[   
\mbox{Moreover,~}\ka_p(\widetilde{T})\leq \ka_p(\widetilde{T}_1)+\sum_{n=2}^\iy (\la+1)\e/2^{n-1}\leq (\la+\e)\ka_p(T)+(\la+1)\e.
\]
Since $\e>0$ is arbitrary, $(b)$ follows. 
\end{proof}

\section{Extension of $T\in K_p^d(X,Y)$}

Similar to Section~2, in this section we assume that $X, Y$ are Banach spaces. Let $(A,\al)$ be an operator ideal.
Recall the definition of the dual operator ideal $(A^d,\al^d)$ with respect to the spaces $X, Y$ as discussed in Section~2. We now recall the following result from \cite[Theorem 2.13]{JD} which will be required to derive our next observation. In this section, we mean $1\leq p\leq\iy$ when no choice of $p$ is mentioned.

\begin{thm}\cite[Theorem 2.13]{JD}\label{T8}
Let $1\leq p<\iy$, let $X$ and $Y$ be Banach spaces and $K$ be $w^*$-compact norming subset of $B_{X^*}$. For every operator $T:X\ra Y$, the following are equivalent:
\bla
\item $T$ is $p$-summing.
\item There exists a regular Borel probability measure $\mu$ on $K$, a closed subspace $X_p$ of $L_p(\mu)$ and an operator $\hat{T}:X_p\to Y$ such that 
\bln
\item $j_pi_X(X)\ci X_p$ and 
\item $\hat{T}j_pi_X(x)=Tx$ for all $x\in X$. In other words, the following diagram commutes.
\el
\begin{center}

\begin{tikzcd}
X \arrow[r, "T"] \arrow[d, "i_X"']        & Y                                         \\
i_X(X) \arrow[d, hook] \arrow[r, "j_p^X"] & X_p \arrow[u, "\hat{T}"'] \arrow[d, hook] \\
C(K) \arrow[r, "j_p"]                     & L_p(\mu)                                 
\end{tikzcd}
\end{center}

\item There exists a probability space $(\Omega,\Sigma,\mu)$ and operators $\hat{T}:L_p(\mu)\ra \ell_{\iy}(B_{Y^*})$ and $v:X\ra L_{\iy}(\mu)$ such that the following diagram commutes.

\begin{center}
\begin{tikzcd}
X \arrow[rr, "T"] \arrow[dd, "v"] &  & Y \arrow[rd, "i_Y", shift left]                &                                 \\
                                  &  &                                                &    \quad \ell_{\infty}(B_{Y^*}) \\
L_{\infty}(\mu) \arrow[rr, "i_p"] &  & L_p(\mu) \arrow[ru, "\hat{T}"', shift right=2] &                                
\end{tikzcd}
\end{center}

In addition, we may arrange $v$ such that $\|v\|=1$ and $\hat{T}$ such that $\|\hat{T}\|=\pi_p(T)$.
\el
\end{thm}

We now derive a few extension properties of $p$-summing operators, where in some cases we also extend the range spaces. Note that in the above diagram, $i_p$ is $p$-summing and $\pi_p (i_p)=1$ (see \cite[p.40]{JD}). Recall that we can factor $j_p$ using canonical mappings:
\begin{tikzcd}
C(K) \arrow[r, "j_{\infty}"] & L_{\infty}(\mu) \arrow[r, "i_p"] & L_p(\mu)
\end{tikzcd}  

\begin{thm}\label{T2}
Let $Y$ be a $P_{\la}$-space.
\bla
\item Suppose that $T\in \Pi_p(X,Y)$. Then for any Banach space $Z\supseteq X$ there exists $\widetilde{T}\in \Pi_p(Z,Y)$ with $\pi_p(\widetilde{T})\leq \la \pi_p(T)$.
\item Suppose that $T\in \Pi_p(Y,X)$. Then for any Banach space $Z\supseteq Y$ there exists $\widetilde{T}\in \Pi_p(Z,X)$ with $\pi_p(\widetilde{T})\leq \la \pi_p(T)$.
\el
\end{thm}
\begin{proof}
$(a).$ Using the decomposition of $T$ as in Theorem~\ref{T8}(b) and because $Y$ is a $P_{\la}$-space, the operator $\hat{T}$ admits an extension $T': L_p(\mu)\ra Y$ with $\|T'\|\leq \la\|\hat{T}\|$. Now, consider the decomposition in Theorem~\ref{T8}(c) and take the norm-preserving extension $\widetilde{v}:Z\to L_{\iy}(\mu)$ of $v$. Then, $\widetilde{T}=T'\circ i_p\circ\widetilde{v}$ is the desired extension, and finally

$\pi_p(\widetilde{T})=\pi_p(T'\circ i_p\circ\widetilde{v})\leq \|T'\|\pi_p(i_p)\|\widetilde{v}\|\leq \la \|\hat{T}\|=\la\pi_p(T)$.

$(b).$ This is obvious.
\end{proof}

\begin{cor}\label{C3}
Let $Y$ be a $P_1$ space and $X$ be any Banach space. Then, for any $T\in\Pi_p(X,Y)$ and $Z\supseteq X$, there exists an extension $\widetilde{T}\in\Pi_p(Z,Y)$ with $\pi_p(T)=\pi_p(\widetilde{T})$, such that $Z\supseteq X$. 
\end{cor}

One may obtain a similar extension property for operators $T\in K_p^d(X,Y)$. However, in this case, we may not have a  $\ka_p^d$-norm preserving extension. 

\begin{thm}\label{T4}
Let $T\in K_p^d(X,Y)$. Then for $\e>0$ and $Z\supseteq X$ there exists an extension $\widetilde{T}\in K_p^d(Z,\ell_{\iy}(B_{Y^*}))$ such that $\ka_p^d(\widetilde{T})\leq \ka_p^d(T)+\e$.
\end{thm}
\begin{proof}
As $T\in K_p^d(X,Y)$, there exists
a Banach space $W$, a $U\in K(X,W)$, and $S\in \Pi_p(W,Y)$ such that $T=SU$ (see \cite[Theorem 3.1]{SK1}).

Let $\e>0$ and $Z\supseteq X$. Choose $\delta=\frac{\e}{\pi_p(S)}$ for some $S$, where $T=SU$ as above.

For this $\de$, there exists $V\supseteq W$ such that $U$ has a compact extension $\widetilde{U}:Z\ra V$ with $\|\widetilde{U}\|\leq \|U\|+\delta$ (see \cite[Theorem 2.3]{JL}). 
By Corollary~\ref{C3}, there exists an extension $\widetilde{S}\in \Pi_p\big(V,\ell_{\iy}(B_{Y^*})\big)$ such that $\pi_p(\widetilde{S})=\pi_p(S)$. Define $\widetilde{T}=\widetilde{S}\circ\widetilde{U}$. Then, 
\beqa
\ka_p^d(\widetilde{T})\leq\inf \Big\{\pi_p(\widetilde{S})\|\widetilde{U}\|:\widetilde{T}=\widetilde{S}\widetilde{U}\Big\}
&\leq & \inf\Big\{\pi_p(S)(\|U\|+\delta):T=SU\Big\}\\
&\leq & \ka_p^d(T)+\delta\inf\big\{\pi_p(S):T=SU\big\}\\
&\leq & \ka_p^d(T)+\e.
\eeqa
Thus, $\widetilde{T}$ is the desired extension and this completes the proof. 
\end{proof}

Now we establish a sufficient condition on $X$ such that any $T\in K_p^d(X,Y)$ (or $K_p^d(Y,X)$) has an extension $\widetilde{T}\in K_p^d(Z,Y)$ (or $K_p^d(Z,X)$) where $Z$ is a Banach space that contains $X$ (or $Y$). 

\begin{thm}\label{T3}
Let $X, Y, Z$ be Banach spaces with $Z\supseteq X$, and $X^{**}$ a $P_\la$ space, for some $\la\geq1$. If $1\leq p\leq\iy$ and $T\in K_p^d(X,Y)$ then there exists $\widetilde{T}\in K_p^d(Z,Y)$ such that $\ka_p^d(\widetilde{T})\leq \la\ka_p^d(T)$.
\end{thm}
\begin{proof}
From \cite[Theorem~3.1]{SK1} there exist a Banach space $W$, compact operator $V\in K(X,W)$, and linear operator $U\in \Pi_p(W,Y)$ such that $T=U\circ V$. 

By \cite[p.11]{JL}$(1)\Ra (6)$, there exists $\widetilde{V}\in K(Z,W)$ such that $\|\widetilde{V}\|\leq\la\|V\|$. Define $\widetilde{T}=U\circ \widetilde{V}$. From \cite[Theorem~3.1]{SK1}, we obtain $\widetilde{T}\in K_p^d(Z,Y)$. Now we estimate $\ka_p^d(\widetilde{T})$: 
$$
\begin{aligned}
\ka_p^d(\widetilde{T}) &\leq \inf\{\pi_p(U).\|\widetilde{V}\|:\widetilde{T}=U\widetilde{V}\text{ as above}\} ~(\mbox{see \cite[Theorem 3.1]{SK1}})\\
&\leq \inf\{\pi_p(U).\la\|V\|:T=UV\}\\
&=\la \ka_p^d(T).
\end{aligned}$$
Therefore, we obtain $\ka_p^d(\widetilde{T})\leq\la\ka_p^d(T)$.
\end{proof}
\begin{thm}
Let $X$ be a $P_{\la}$-space and and $T\in K_p^d(Y,X)$. Then for any $Z\supseteq Y$, there exists $\widetilde{T}\in K_p^d(Z,X)$ with $\ka_p^d(\widetilde{T})\leq (\la+\e)\ka_p^d(T)$.    
\end{thm}
\begin{proof}
Since $T\in K_p^d(Y,X)$, there exist
a Banach space $W$, $V\in K(Y,W)$ and $U\in \Pi_p(W,X)$ such that $T=UV$ (see \cite[Theorem 3.1]{SK1}).

Now let $0<\e'<\frac{\e\ka_p^d(T)}{\la\pi_p(U)}$ for some $U$ such that $T=UV$ as above. If $Z\supseteq Y$ then there exists $E\supseteq W$ such that $V$ has a compact extension $\widetilde{V}:Z\ra E$ with $\|\widetilde{V}\|\leq \|V\|+\e'$ (see \cite[Theorem 2.3]{JL}).

Now by Theorem~\ref{T2}, $U$ has an extension $\widetilde{U}\in\Pi_p(E,X)$ with $\pi_p(\widetilde{U})\leq\la\pi_p(U)$. The desired extension is $\widetilde{T}=\widetilde{U}\widetilde{V}$. Using \cite[Theorem 3.1]{SK1}, we estimate $\ka_p^d(\widetilde{T})$.
$$
\begin{aligned}
\ka_p^d(\widetilde{T}) &\leq \inf\{\pi_p(\widetilde{U}).\|\widetilde{V}\|:\widetilde{T}=\widetilde{U}\widetilde{V}\text{ as above}\}\\
&\leq \inf\{\la.\pi_p(U).(\|V\|+\e'):T=UV\}\\
&\leq \la\ka_p^d(T)+\la\e'\inf\big\{\pi_p(U):T=UV\big\}\\
&<(\la+\e)\ka_p^d(T).
\end{aligned}$$
This completes the proof.
\end{proof}

\begin{thm}\label{T6}
Let $X$ be a real Banach space, $p>1$ and $\e>0$. Suppose that for all $q\geq p$ and every operator $T:Y\ra X$ with $\dim T(X) \leq 3$ has an extension $\widetilde{T}:Z\ra X$ exists, $Z\supseteq Y$ with $\dim Z/Y=1$ and $\ka_q^d(\widetilde{T})\leq (1+\e)\ka_q^d(T)$. Then $X$ is an $L_1$-predual.
\end{thm}

\begin{proof}
From the assumption on $T$, it follows from Theorem~\ref{T1} that $\|\widetilde{T}^*\|\leq (1+\e)\|T^*\|$. The result now follows from \cite[Theorem~5.4]{JL}.
\end{proof}

\begin{prop}\label{P5}
For Banach spaces $X, Y$ \tFAE.
    \bla
\item For every $T\in K_p^d(Y,X)$, $Z\supseteq Y$ and $\e>0$, there exists  $\widetilde{T}\in K_p^d(Z,X)$ with $\ka_p^d(\widetilde{T})\leq (\la+\e)\,\ka_p^d(T)$ and $\ka_p^d(\widetilde{T}|_{Y}-T)\leq\e$.
\item For every $T\in K_p^d(Y,X),Z\supseteq Y$ and $\e>0$, there exists extension $\widetilde{T}\in K_p^d(Z,X)$ with $\ka_p^d(\widetilde{T})\leq(\la+\e)\,\ka_p^d(T)$.
    \el
\end{prop}
\begin{proof}
It remains to prove $(a)\Ra (b)$. We follow similar techniques used in the proof of Proposition~\ref{P1}.

Using similar arguments stated in Proposition~\ref{P1}, we get a sequence $(\widetilde{T}_n)\ci K_p^d(Z,X)$  satisfying 
\begin{equation}\label{eq10}
   \ka_p^d(\widetilde{T}_1)\leq\,(\la+\e)\,\ka_p^d(T),\quad \ka_p^d(\widetilde{T}_1|_Y-T)<\frac{\e}{2}
\end{equation}
for $n=1$, and for $n\geq 2$, we have
\[
    \ka_p^d(\widetilde{T}_n)\leq\,(\la+1)\,\ka_p^d\Big(T-(\widetilde{T}_1+\widetilde{T}_2+...+\widetilde{T}_{n-1})|_Y\Big),
\]
\begin{equation}\label{eq11}
    \ka_p^d\Big(\widetilde{T}_n|_Y-\big(T-(\widetilde{T}_1+\widetilde{T}_2+...+\widetilde{T}_{n-1})|_Y\big)\Big)<\frac{\e}{2^n}.
\end{equation}

For $n\geq 2$, we have $\ka_p^d(\widetilde{T}_n)\leq (\la+1)\e/2^{n-1}$. Hence, the series $\sum\limits_{n=1}^\iy \widetilde{T}_n$ converges in the $\ka_p^d$ norm topology to an operator $\widetilde{T}\in K_p^d(Z,X)$ satisfying $\widetilde{T}|_Y=T$. In fact, from equation~\ref{eq11} for every $\de>0$, there exists $m$ such that $\ka_p^d((T-(\widetilde{T_1}+\widetilde{T_2}+\ldots+\widetilde{T_m})|_Y))<\de$.
   \[   
\mbox{Moreover,~}\ka_p^d(\widetilde{T})\leq \ka_p^d(\widetilde{T}_1)+\sum_{n=2}^\iy (\la+1)\e/2^{n-1}\leq (\la+\e)\ka_p^d(T)+(\la+1)\e.
\]
Since $\e>0$ is arbitrary, $(b)$ follows. 
\end{proof}

\begin{thm}\label{T10}
\bla
\item Let $X$ be a Banach space such that every $T\in K_p^d(Y,X)$ has an extension $\widetilde{T}\in K_p^d(Z,X)$, where $(Z\supseteq Y)$. Then there exists a constant $\eta$ such that for every such $Y,Z$, and $T$ there exists a $p$-compact extension $\widetilde{T}$ with $\ka_p^d(\widetilde{T})\leq \eta\ka_p^d(T)$.
\item Let $X$ be a Banach space such that every $T\in K_p^d(X,Y)$ has an extension $\widetilde{T}\in K_p^d(Z,Y)$, where $(Z\supseteq X)$. Then there exists a constant $\eta$ such that for every such $Y,Z$, and $T$ there exists a $p$-compact extension $\widetilde{T}$ with $\ka_p^d(\widetilde{T})\leq \eta\ka_p^d(T)$.
\el
\end{thm}
\begin{proof}
$(a).$ The proof proceeds in the same manner as that of Theorem~\ref{T9}. Thus, it suffices to prove that $\ka_p^d(T_n')=\ka_p^d(T_n)$, where $T,T_n',T_n,Y$ and $Y_n$ are as in theorem~\ref{T9}. 
      
Therefore, we need to prove that $\ka_p((T_n')^*)=\ka_p(T_n^*)$. First we observe that $Y_n^*\supseteq T_n^*(X^*)\cong\big((T_n')^*\big)(X^*)\subseteq(0,...,Y_n^*,0,...)\cong Y_n^*$. In particular for $x^*\in X^*$ and $y=(y_1,...,y_n,...)\in Y$, $T_n'^*(x^*)(y)=x^*(T_n'(y))=x^*(T_ny_n)=(T_n^*)(x^*)(y_n)=\Big((0,...,(T_n^*)(x^*),0,...\Big)(y)$. 

In this way, we observe that both the sets $T_n^*(B_{X^*})$ and $T_n'^*(B_{X^*})$ are the same. It follows that $\ka_p(T_n'^*)=\ka_p(T_n^*)$. 

$(b).$ We first claim the following.

{\sc Claim:~} There exists a $P_1$-space $W$ such that for any Banach space $Y$ and
$T \in K_p^d(X,Y)$ there exists an extension $\widetilde{T} \in K_p^d(W,Y)$ with
$\ka_p^d(\widetilde{T}) \leq \eta \ka_p^d(T)$.

Suppose no such $\eta$ exists. Then for every $n$ there exist a $P_1$ space
$W_n \supseteq X$, a Banach space $Y_n$ and $T_n \in K_p^d(X,Y)$ with $\ka_p^d(T_n) = 1$ such that any
extension $\widetilde{T}_n\in K_p^d(W_n,Y_n)$ of $T_n$ satisfies
$\ka_p^d(\widetilde{T}_n) \geq n^3$. Let $Y=\oplus_{c_0}Y_n$ and consider $T_n:X\to Y$ as each $Y_n$ is a subspace of $Y$. Now define $T : X \to Y \text{ by } T = \sum_{n=1}^\iy \frac{T_n}{n^2}$. Clearly, $T \in K_p^d(X,Y)$ since each $T_n \in K_p^d(X,Y)$. Now observe that
$\bigoplus_{\ell_\iy} W_n = W$ is a $P_1$ space as each $W_n$ is a $P_1$ space and
$W$ contains $X$. By our hypothesis there exists an extension
$\widetilde{T} \in K_p^d(W,Y)$ of $T$. Then the restriction of $n^2\widetilde{T}$ to
$W_n$ is an extension of $T_n$. From our assumption
$\ka_p^d(n^2\widetilde{T}) \geq n^3$, which leads to
$\ka_p^d(\widetilde{T}) \geq n$, for all $n$, which is a contradiction.

Next, assume that $Z$ is a Banach space and $Z \supseteq X$ and let
$T \in K_p^d(X,Y)$. From the above claim, there exists an extension
$\widetilde{T} \in K_p^d(Z,Y)$. Moreover, the identity
$I : X \to X$ has an extension $\widetilde{I} : Z \to W$ with $\|\widetilde{I}\| = 1$,
which follows from the property of $P_1$-space. Clearly
$\widetilde{T} \circ \widetilde{I}: Z \ra Y$ is an extension of $T$. Moreover,
$\widetilde{T} \circ \widetilde{I} \in K_p^d(Z,Y)$ and $
\ka_p^d(\widetilde{T} \circ \widetilde{I}) \leq \ka_p^d(\widetilde{T}) \|\widetilde{I}\|
\leq \eta \ka_p^d(T)$.
  \end{proof}

Declaration of interest: All authors declare that they have no conflicts of interest.

\end{document}